\documentclass[12pt,twoside]{amsart}
\usepackage[all]{xy}
\usepackage {amssymb,latexsym,amsthm,amsmath,mathtools, multirow, amsfonts, longtable, eqnarray, hyperref,comment,fancyhdr}
\usepackage{enumitem,color}

\hypersetup{
	colorlinks=true,
	linkcolor=blue,
	filecolor=magenta,      
	urlcolor=cyan,
	citecolor=red,
}

\long\def\symbolfootnote[#1]#2{\begingroup%
	\def\thefootnote{\fnsymbol{footnote}}\footnote[#1]{#2}\endgroup}

\newcommand{\Z}{\ensuremath{\mathcal{Z}}}

\newcommand{\psl}{\textup{PSL}}

\makeatletter
\def\imod#1{\allowbreak\mkern10mu({\operator@font mod}\,\,#1)}
\makeatother

\newtheorem{theorem}{Theorem}[section]
\newtheorem{lemma}[theorem]{Lemma}

\newtheorem{proposition}[theorem]{Proposition}
\newtheorem*{theorem*}{Theorem}
\theoremstyle{definition}

\newtheorem{remark}[theorem]{Remark}
\newtheorem{question}[theorem]{Question}
\newtheorem{example}[theorem]{Example}
\newtheorem{conjecture}[theorem]{Conjecture}
\numberwithin{equation}{section}
\newcommand{\ignore}[1]{}

\newcommand{\mynote}[1]{}
\begin{document}

	\title{Counterexamples to a conjecture of M. Pellegrini and P. Shumyatsky}
	\author{Rijubrata Kundu}
	\address{Indian Institute of Science Education and Research Mohali, Knowledge City, Sector 81, Mohali 140 306, India}
	\email{rijubrata8@gmail.com}
	\author{Sumit Chandra Mishra}
	\address{Indian Institute of Science Education and Research Mohali, Knowledge City, Sector 81, Mohali 140 306, India}
	\email{sumitcmishra@gmail.com}
	\thanks{The authors would like to acknowledge the support of IISER Mohali institute postdoctoral fellowship during this work}
	%\date{}
	\today 
	\begin{abstract}
		In this article, we provide counterexamples to a conjecture of M. Pellegrini and P. Shumyatsky which states that each coset of the centralizer of an involution in a finite non-abelian simple group $G$ contains an odd order element, unless $G=\psl(n,2)$ for $n\geq 4$. More precisely, we show that the conjecture does not hold for the alternating group $A_{8n}$ for all $n\geq 2$.
	\end{abstract}
	
	\maketitle
	
	\section{Introduction}
	
	The subgroup of a group $G$ generated by the set of commutators $\{[a,b]=aba^{-1}b^{-1}\mid a,b\in G\}$ in $G$ is called the commutator subgroup of $G$, and is usually denoted by $[G,G]$ or $G'$. The subgroup $G'$ is  normal in $G$, and therefore for a non-abelian simple group $G$, we have $G'=G$. It was proved by O. Ore in \cite{ore} that any element in the alternating group $A_n$ $(n\geq 5)$ is a single commutator. Further, he conjectured that any element of a finite non-abelian simple group is a single commutator. This conjecture, called the Ore's conjecture, is now solved in the affirmative owing to the contributions of several mathematicians (see \cite{lost}). We refer the readers to an excellent article by G. Malle (see \cite{ma} and the references therein) on Ore's conjecture.
	
	\medskip
	
	For a finite group $G$, an element $g\in G$ is called a coprime commutator if $g=[x,y]$ where $x,y\in G$ are elements whose orders are coprime. An element $x\in G$ is called an involution if its order is two. The authors in \cite{ps} proved the following:
	
	\begin{theorem}[Theorem 1.1, \cite{ps}]
		Let $q>3$ be a prime-power. Then every element of $\psl(2,q)$ is a coprime commutator.
	\end{theorem}
	In the same article, they conjecture that the above holds true for all finite non-abelian simple groups. The following result played a crucial role in the proof of the above theorem.
	
	\begin{theorem}[Theorem 1.2,  \cite{ps}]
	Let $q>3$ be a prime-power. Then every coset of the centralizer of an involution in $\psl(2,q)$ has an odd order element.
	\end{theorem}

	In the light of the above result, they further conjectured the following:
	
	\begin{conjecture}\label{mainconjecture}[Conjecture 1.4,  \cite{ps}]
		Each coset of the centralizer of an involution in a finite non-abelian simple group has an odd order element, unless $G=\psl(n,2)$ for $n\geq 4$.
	\end{conjecture}

	They showed that $G=\psl(n,2)$ for $n\geq 4$ is indeed an exception to the above conjecture. In \cite{sh}, the author showed that every element in the alternating group $A_n$ $(n\geq 5)$ is a coprime commutator by direct computations. Recently, G. Zini (see \cite{zi}) has verified Conjecture~\ref{mainconjecture} for the Suzuki groups $^2B_2(q)$ where $q=2^{2n+1}$ $(n\geq 1)$. Further, he uses this result to prove that every element in the simple group $^2B_2(q)$ is a coprime commutator.
	
	In this article, we consider the alternating groups $A_{8n}$ $(n\geq 1)$. For each of these groups, we exhibit an involution and a coset of the centralizer of that involution consisting entirely of even order elements. This provides an infinite family of counterexamples to Conjecture~\ref{mainconjecture}. The main theorem is as follows:
	\begin{theorem}\label{main-theorem}
		Let $N=8n$,  $n\geq 1$. Let $x=(1\;2)(3\;4)\cdots(4n-1\;4n)\in A_N$. Let $\mathcal{Z}_{A_N}(x)$ denote the centralizer of $x$ in $A_N$. Let $y=(1\;4n+1)(2\;4n+2)\cdots(4n\;8n)\in A_N$. Then each element in the coset $y\mathcal{Z}_{A_N}(x)$ is a product of disjoint cycles, each of even length. In particular, each element in the coset $y\mathcal{Z}_{A_N}(x)$ has even order. 
	\end{theorem}

	\begin{remark}
		A permutation $\pi$ has even order if and only if it contains an even length cycle in its disjoint cycle decomposition. Thus, in order to disprove the conjecture we just need to show that every element in $y\mathcal{Z}_{A_N}(x)$ has  at least one cycle of even length in its disjoint cycle decomposition. However, in the above theorem we prove a much stronger statement that each element in the  coset $y\mathcal{Z}_{A_N}(x)$ is a product of disjoint cycles, each of even length.
	\end{remark}
	
	The proof of the above theorem is in the next section. In the final section, we make some comments on another related conjecture of Pellegrini and Shumyatsky. We also comment on a related question of Zappa in the same section.
	\section{Proof of the main theorem}
	
	We fix some notations which will be used throughout the paper. Let $N=8n$, $n\geq 1$. Let $[N]$ denote the set $\{1,2,\ldots,N\}$. We consider the alternating group $A_N$ on $[N]$.  Let $x=(1\;2)(3\;4)\cdots(4n-1\;4n) \in A_N$. Let $\mathcal{Z}_{A_n}(x)$ be the centralizer of $x$ in $A_N$. Take $y=(1\;4n+1)(2\;4n+2)\cdots(4n\;8n) \in A_N$. We divide $[N]$ into two disjoint subsets $A=\{1,2,\ldots, 4n\}$ and  $B=\{4n+1,4n+2,\ldots,8n\}$ each of size of $4n$. We say that a permutation $\pi \in S_N$ fixes $A$ if for all $a\in A$, $\pi(a)\in A$. This immediately implies $\pi$ fixes $B$ as well. Products of permutations will be executed from right to left. 
	
	\medskip
	
	 We start with a general lemma.
	
	\begin{lemma}\label{lemma}
		Let $z\in A_N$ be such that $z$ fixes $A$. Then $\tau=yz$ is a product of disjoint cycles each of even length.
	\end{lemma}
	
	\begin{proof}
		Let $z\in A_{N}$ be as in the statement of the lemma. We note that $y(a)\in B$ for all $a\in A$, and $y(b)\in A$ for all $b\in B$. The same holds for $\tau$ since $z$ fixes $A$ and $B$. Let $a_1$ be any element of $A$. Then $\tau(a_1)=b_1$ for some $b_1\in B$. We know that $\tau(b_1)\in A$. If $\tau(b_1)=a_1$, then the cycle containing $a_1$ in the disjoint cycle decomposition of $\tau$ is $(a_1\;b_1)$ which is of even length and we are done. If not, then $\tau(b_1)=a_2$ for some $a_2\in A\setminus\{a_1\}$. Then $\tau(a_2)=b_2$ for some $b_2 \in B\setminus\{b_1\}$. If $\tau(b_2)=a_1$, then the cycle containing $a_1$ is $(a_1\;b_1\;a_2\;b_2)$ which is of even length and we are done. If not, we proceed as before and after a finite number of steps, the cycle containing $a_1$ in the disjoint cycle decomposition of $\tau$ is of the form $(a_1\;b_1\;\cdots\; a_k\;b_k)$ for some $k$, $1\leq k\leq 4n$, which is of even length. Similarly for any $b\in B$, the cycle containing $b$ in the disjoint cycle decomposition of $\tau$ is of even length. This completes the proof.
	\end{proof}
	
	\begin{proof}[Proof of Theorem~\ref{main-theorem}]
		Let $z\in \mathcal{Z}_{A_N}(x)$. We claim that $z$ fixes $A$. Suppose, on the contrary, that $z$ does not fix $A$. 
		Then there exists $a \in A$ such that  $z(a) = b$ for some $b \in B$. Then,
		\begin{eqnarray*}
		&& \;\;\;\;\;\;\;\;\; z^{-1}(b)= a \implies xz^{-1}(b)=a' \; \text{for some $a'\in A\setminus\{a\}$} \\ && \implies zxz^{-1}(b)=z(a') \implies x(b)=z(a') \implies b=z(a').
		\end{eqnarray*}
	This is a contradiction as $a'\neq a$ and $z(a)=b$. This proves our claim that $z$ fixes $A$. The proof of the theorem now follows from Lemma~\ref{lemma}.
	\end{proof}
	
	\begin{remark}
		Both Lemma~\ref{lemma} and Theorem~\ref{main-theorem} remains true if we replace $A_N$ by $S_N$, the symmetric group on $[N]$.
	\end{remark}

	\begin{remark}
		Note that $A_8 \cong \psl(4,2)$ which already appears in Conjecture~\ref{mainconjecture} as an exception. For $n\geq 2$, $A_{8n}$ is not isomorphic to any of the simple groups $\psl(m,2)$ for $m\geq 4$ and hence is a counterexample to the conjecture.
	\end{remark}

	We write an example for the purpose of clarity. Before proceeding with the example we fix a notation. Given $n\in \mathbb{N}$, we denote the cycle type of a permutation $\pi\in S_n$ as $1^{c_1}2^{c_2}\cdots$, where $c_i$ denotes the number of $i$-cycles in the disjoint cycle decomposition of $\pi$. Clearly, $1^{c_1}2^{c_2}\cdots $ satisfies $\sum_{i}ic_i=n$, whence $1^{c_1}2^{c_2}\cdots$ is a partition of $n$.
	\begin{example}
		For $N=8$, we have the following: $x=(1\;2)(3\;4) \in A_8$ and $y=(1\;5)(2\;6)(3\;7)(4\;8)\in A_8$. Then, the centralizer is given by $$\mathcal{Z}_{A_8}(x)=\langle(1\;3)(2\;4), (1\;2)(3\;4), (1\;2)(5,8), (1\;2)(6\;8), (1\;2)(7\;8) \rangle.$$
		We have $|\mathcal{Z}_{A_8}(x)|=96$. We used GAP to get the list of all 96 elements in the coset $y\mathcal{Z}_{A_8}(x)$. The cycle type of the elements occurring in $y\mathcal{Z}_{A_8}(x)$ belongs to the set  $\{2^4, 2^16^1, 4^2\}$ of partitions of 8. We conclude that the elements in $y\mathcal{Z}_{A_8}(x)$ have orders 2,4, and 6. 
	
	For $N=16$, we have the following: $x=(1\;2)(3\;4)(5\;6)(7\;8)\in A_{16}$ and $y=(1\;9)(2\;10)(3\;11)(4\;12)(5\;13)$ $(6\;14)(7\;15)(8\;16)\in A_{16}$. We have $|\mathcal{Z}_{A_{16}}(x)|=7741440$. Using GAP we see that the cycle type of the elements occurring in $y\mathcal{Z}_{A_{16}}(x)$ belongs to the set  
	$$X=\{2^8, 2^44^2, 2^26^2, 2^56^1, 2^24^18^1, 2^310^1, 6^110^1, 4^112^1, 2^114^1, 2^14^26^1, 8^2, 4^4\},$$ of partitions of 16. It immediately follows that the orders of elements  in $y\mathcal{Z}_{A_{16}}(x)$ belongs to the set $\{2,4,6,8,10,12,14,30\}$.
	\end{example}

	\section{Remarks on related questions}
	
	In this final section, we comment on a related conjecture in \cite{ps}, and a question of G. Zappa (see \cite{za}).
	
	\subsection{A related conjecture of Pellegrini and Shumyatsky}
	
	In \cite{ps}, the authors also conjectured the following:
	\begin{conjecture}[Conjecture 1.3, \cite{ps}]\label{conjecture2}
		Let $T$ be a Sylow-2 subgroup of a finite group $G$ and $t$ an involution in $Z(T)$. Then each coset of the centralizer $\Z_G(t)$ contains an odd order element.
	\end{conjecture}

	We mention that the alternating group $A_4$, $A_{8}$ and $A_{16}$ are not counterexamples to the above conjecture. This can be easily verified using GAP. For example, when $G=A_4$, $T=\{id,(1\;2)(3\;4),(1\;3)(2\;4),(1\;4)(2\;3))\}$. There are two non-trivial cosets of the  centralizer $\Z_G((1\;2)(3\;4))$ each of which consists of four 3-cycles. For the groups $A_8$ and $A_{16}$ one needs to compute the cosets of the centralizers $\Z_{A_8}((1\;2)(3\;4)(5\;6)(7\;8))$ and $\Z_{A_{16}}((1\;2)(3\;4)(5\;6)(7\;8)(9\;10)(11\;12)(13\;14)(15\;16))$ in $A_8$ and $A_{16}$ respectively. In each of these case, we observed that every non-trivial cosets contains an odd order element. In fact, every non-trivial coset contains a longest odd length cycle in each of the two cases, that is, 7-cycle in the first case and 15-cycle in the second case.
	
	\subsection{A related question of Zappa}
	
	Let $p$ be a prime. In 1962, G. Zappa (see \cite{za}) posed the following question:
	
	\begin{question}\label{question1}
		Let $G$ be a finite group and $P$ be a Sylow $p$-subgroup of $G$. Can a non-trivial coset $gP$ contain only elements whose orders are powers of $p$? And in that case, can at least one element of $gP$ have order $p$?
	\end{question}

	This question was asked again in \cite{gg}. In \cite{co}, the author has answered this question in the positive for $p=5$ and $|P|=5$. He showed that if $G=\psl(3,4)$ and $P$ be a Sylow $5$-subgroup of $G$, then there exists a non-trivial coset of $P$ whose every element has order 5. This is the smallest example which gives a positive answer to the above question. He further proves that for a finite group $G$, if $|P|=2,3,4,8$, then $G$ cannot give a positive answer to the above question. Thus, $|P|\geq 16$ if $p=2$ and $|P|\geq 5$ if $p$ is odd. It is still an open question whether other examples with the above restrictions on $P$ exists. 
	
	\medskip
	
	It is easy to see that if $G$ is the smallest finite group for a prime $p$ which gives a positive answer to Question~\ref{question1}, then $G$ is a non-abelian finite simple group (see Lemma 4, \cite{co}). The following is proved in \cite{co}.
	\begin{theorem}[Theorem 2, \cite{co}]\label{counter-example-PSL} Let $G$ be a finite group which gives a positive answer to Question~\ref{question1} for a prime $p$. Then $G$ cannot be $\psl(2,q)$ for any prime-power $q$.
	\end{theorem}
	In the remaining part of this article, we prove a similar partial result for the alternating groups $A_n$. For an element $x \in S_n$, let $supp(x)$ denote the support of $x$. 
	
	\begin{lemma}\label{Lemma_tau}
		Let $m \geq n$ and $x$ be a $n$-cycle in $S_m$. Then the cyclic subgroup $P$ generated by $x$ in $S_m$ acts regularly on $supp(x)$.
	\end{lemma}
	
	\begin{proof}
		The proof is obvious for $n=1,2.$ So assume $ n \geq 3$. 
		Let $a_1$ and $a_2$ be distinct elements of $supp(x)$. We can write $x$ as 
		$x = (a_1\;b_1\; \dots\;b_{n-1})$ where  
		$b_j \in A$, $1 \leq j \leq n-1$. 
		Then $a_2 = b_k$ for some $k, 1 \leq k \leq n-1$ 
		and $\tau=x^k$ maps $a_1$ to $a_2$. The fact that $\tau$ is unique is obvious.
	\end{proof}
	
	For a prime $p$, let $P$ be a Sylow $p$-subgroup of $A_n$ of order $p$. Then $p\leq n\leq 2p-1$, and in that case a Sylow $p$-subgroup is generated by a $p$-cycle. We have the following:
	\begin{proposition}\label{sylow-p-p}
	Let $p\geq 3$ be a prime. Let $n\in \mathbb{N}$ and $p \leq n \leq 2p-1$. Then for $A_n$ and any Sylow $p$-subgroup $P$ of $A_n$, there does not exist a non-trivial coset of $P$ consisting entirely of $p$-cycles.
	\end{proposition}

	\begin{proof}
			It is enough  to prove this for $n=2p-1$. 
			Also since all Sylow $p$-subgroups are conjugate, 
			it is enough to prove it for a particular Sylow $p$-subgroup $P$. Let $P$ be the subgroup generated by the $p$-cycle $(1\;2\;3\; \dots \; p)$. Let $A$ denote the set $\{1, 2, \dots, p\}$. It is enough to consider the left cosets $xP$ where $x\notin P$ is a $p$-cycle. Let $B=supp(x)$. We will show that there is an element $\tau \in P$ such that $x \tau$ is not a $p$-cycle. We do this case-by-case.
			
			\medskip
			
			\textbf{\underline{Case I}:} Let $A=B$. Then we can write $x = (1\;a_1\; a_2\; \dots\; a_{p-1})$ 
			where $a_i$'s are distinct elements of $A\setminus \{1\}$. 
			By Lemma \ref{Lemma_tau}, we can choose an element $\tau \in P$ which maps $a_1$ to $1$. 
			Then $x \tau$ fixes $a_1$ and hence is not a $p$-cycle.
			
			\textbf{\underline{Case II}:} Suppose that $|A \cap B|=1$. Then $x=(X_1 \;\dots\;X_{p-1}\;a)$ where $X_i$'s are not in $A$ and $a\in A$. Let $a_1$ be an element in $A\setminus \{a\}$. Using Lemma~\ref{Lemma_tau} we can choose a $\tau\in P$ which maps $a_1$ to $a$. Then $x\tau$ is clearly a $(2p-1)$-cycle and we are done.
			
			\textbf{\underline{Case III}:} Suppose that $2 \leq |A \cap B| \leq p-1$. Without loss of generality, we can choose $x=(X\;\dots \; a)$ where $a\in A$ and $X\notin A$. Let $a_1$ be the first element of $A$ that appears to the right of $X$ in $x$. Then $a_1\neq a$. By Lemma~\ref{Lemma_tau}, we can choose $\tau \in P$ which takes $a_1$ to $a$. Then the cycle of $a_1$ in the disjoint cycle decomposition of $x\tau$ is given by $(a_1\;X\;\dots\;X_t)$ where $X_t$ is the symbol that appears just before $a_1$ in $x$. This cycle has length less than $p$. Thus $x\tau$ is not a $p$-cycle. This completes the proof.
	\end{proof}
	
	Let $n$ be such that $2p\leq n\leq 3p-1$ and $P$ be a Sylow $p$-subgroup of $A_n$. Then $|P|=p^2$ and $P$ is generated by two $p$-cycles whose supports do not intersect. In fact, any non-identity element in $A_n$ whose order is  a power of $p$ is a $p$-cycle or a disjoint product of two $p$-cycles. We have the following:
	\begin{lemma}\label{sylow-p-p^2-p}
		Let $p\geq 3$ be a prime, and  $2p \leq n \leq 3p-1$. Then for any Sylow $p$-subgroup $P$ of $A_n$ and a $p$-cycle $x\notin P$, $xP$ contains an element whose order is not equal to $p$.
	\end{lemma}

	\begin{proof}
		It is enough to prove the result for $A_{3p-1}$. Without loss of generality, we can assume $P=\langle (1\;2\;\cdots\;p),\;(p+1\;p+2\;\cdots\;2p)\rangle$. Let $x\notin P$ be a $p$-cycle. Let $A=\{1,2,\ldots, p\}$, $B=\{p+1,p+2,\ldots, 2p\}$, and $C=\{2p+1,2p+2,\ldots, 3p-1\}$. Let $X$ be the support of $x$. If $X\subseteq A\cup C$ or $X\subseteq B\cup C$ or $X\subseteq A\cup B$ we are done by the previous proposition. If not, $X$ has at least one element from each of the set $A,B,C$. Consider the set $Y=X\cap (B\cup C)$. Since $X$ has at least one element from $A$, $|Y|\leq p-1$. Thus, this case also reduces to the set-up considered in the previous proposition. We conclude that $xP$ contains an element whose  order is not equal to $p$.

	\end{proof}

	\begin{proposition}
		Let $p\geq 3$ be a prime. Let $P$ be a Sylow $p$-subgroup of $A_{2p}$. Then every non-trivial coset of $P$ contains an element whose order is not equal to $p$. 
	\end{proposition}

	\begin{proof}
		
		We take $P=\langle (1\;2\;\cdots\;p),\;(p+1\;p+2\;\cdots\;2p)\rangle$. Further, we set $A=\{1,2,\ldots, p\}, B=\{p+1,p+2,\ldots ,2p\}$. If $x\notin P$ be a $p$-cycle then we are done by the previous lemma.
		
		Let $x$ be a disjoint product of two $p$-cycles. We write $x=yz$ where $y$ and $z$ are $p$-cycles. Let $Y=supp(y)$ and $Z=supp(z)$.
		If $Y=A$ or $Y=B$ (then $Z=B$ or $Z=A$) we are done by the previous lemma. Let us now assume that $Y$ contains at least one element each from $A$ and $B$. The same then holds for $Z$. Since $p$ is odd, the cycle $y$ must have symbols from $A$ which are adjacent or, symbols from $B$ which are adjacent. Without loss of generality, assume $y$ takes $a_1$ to $a_2$ where $a_1,a_2\in A$. Using  Lemma~\ref{Lemma_tau}, there exists $\tau \in \langle (1\;2\;\cdots\;p) \rangle \subseteq P$ which takes $a_2$ to $a_1$. Then $x\tau$ fixes $a_2$ and thus $|supp(x\tau)|<2p$. We conclude that  $x\tau$ is not a disjoint product of two $p$-cycles. If $x\tau$ is a not a $p$-cycle then $x\tau$ cannot have order $p$. If $x\tau$ is a $p$-cycle then we conclude by the previous lemma that $x\tau P=xP$ does not contain an element of order $p$. This completes the proof.
	\end{proof}
		
	Thus, in the spirit of Theorem~\ref{counter-example-PSL}, we have proved that when $P$ is a Sylow $p$-subgroup of order $p$ of a group $G$, then $G$ cannot be $A_n$. Further, if $P$ has order $p^2$, then $G$ cannot be $A_{2p}$. Once again GAP computations indicate that $A_{2p}$ can be replaced by any $A_n$ where $2p\leq n\leq 3p-1$. Partial evidence is provided by Lemma~\ref{sylow-p-p^2-p}. In \cite{co}, the author has given computational evidence and mentions that $|P|=5$ might be the only possibility which gives a positive answer to Question~\ref{question1}. Theorem~\ref{sylow-p-p} provides more evidence in this direction.
	\subsection*{Acknowledgment}
	We thank Prof. Pellegrini and Prof. Shumyatsky for going through an earlier version of the manuscript and providing helpful comments and suggestions. We thank Prof. Amit Kulshrestha and Prof. Chetan Balwe for their support. We also thank Harish Kishnani and Dr. Gurleen Kaur for their interest in this work.
	
\end{document}